\documentclass[11pt]{amsart}

\usepackage{latexsym}
\usepackage{amssymb}
\usepackage{amsmath}
\usepackage{amsthm}
\usepackage{amsfonts}
\usepackage{color}
\usepackage{pictexwd,dcpic}

\usepackage{graphicx}
\usepackage{psfrag}
\usepackage{hyperref}
\usepackage{comment}

\newtheorem{thm}{Theorem}[section]
\newtheorem{cor}[thm]{Corollary}
\newtheorem{lem}[thm]{Lemma}
\newtheorem{prop}[thm]{Proposition}

\newtheorem{theorem}{Theorem}

\newtheorem{corollary}[theorem]{Corollary}

\theoremstyle{definition}
\newtheorem{defn}[thm]{Definition}

\newtheorem{example}[thm]{Example}
\newtheorem{rem}[thm]{Remark}

\newtheorem*{ack}{Acknowledgements}

\numberwithin{equation}{section}

\newcommand{\scal}[1]{\langle #1 \rangle}

\def\P{\ensuremath{\mathsf{p}}}
\DeclareMathOperator{\Diff}{Diff}

\newcommand{\RR}{\mathbb{R}}
\newcommand{\CC}{\mathbb{C}}

\DeclareMathOperator{\diam}{diam}
\DeclareMathOperator{\Iso}{Isom}

\DeclareMathOperator{\codim}{codim}

\newcommand{\CP}{\mathbb{C}P}

\newcommand{\sphere}{\mathrm{\mathbb{S}}}

\newcommand{\ZZ}{\mathbb{Z}}

\newcommand{\fol}{\mathcal{F}}
\newcommand{\SO}{\mathrm{SO}}
\newcommand{\On}{\mathrm{O}}

\def\ol{\overline}

\def\lra{\longrightarrow}

\def\ra{\rightarrow}

\def\In{\subseteq}

\def\CC{\mathbb{C}}

\def\PP{\Bbb{P}}

\def\RR{\mathbb{R}}
\def\ZZ{\mathbb{Z}}

\def\mc{\mathcal}

\def\leaf{\mc{L}}
\def\fol{\mc{F}}

\def\O{\textrm{O}}

\def\codim{\textrm{codim\,}}

\begin{document}

%-----------------------------------------------------------------------------
% BEGIN FRONT MATTER ----------------------------------------------------------
%-----------------------------------------------------------------------------

% TITLE

\title[Singular Riemannian foliations]{Singular Riemannian foliations and  applications to positive and nonnegative curvature}

% AUTHOR 1

\author[F. Galaz-Garcia]{Fernando Galaz-Garcia$^*$}
\address[Galaz-Garcia]{Mathematisches Institut, WWU M\"unster, Germany.}
\curraddr{Institut f\"ur Algebra und Geometrie, Karlsruher Institut f\"ur Technologie (KIT), Karlsruhe, Germany.}
\email{galazgarcia@kit.edu}
\thanks{$^*$ Partially supported by  SFB 878: \emph{Groups, Geometry \& Actions}, at the University of M\"unster.}

% AUTHOR 2

\author[M. Radeschi]{Marco Radeschi$^*$}
\address[Radeschi]{Mathematisches Institut, WWU M\"unster, Germany.}
\curraddr{Institut f\"ur Algebra und Geometrie, Karlsruher Institut f\"ur Technologie (KIT), Karlsruhe, Germany.}
\email{mrade\_02@uni-muenster.de}
\thanks{}

% DATE
\date{\today}

% MATH SUBJECT CLASSIFICATION AND KEYWORDS


% ABSTRACT

\begin{abstract}
We determine the structure of the fundamental group of the regular leaves of a closed singular Riemannian foliation on a compact, simply connected Riemannian manifold. We also study closed singular Riemannian foliations whose leaves are homeomorphic to aspherical or to Bieberbach manifolds. These foliations, which we call \emph{A-foliations} and \emph{B-foliations}, respectively, generalize isometric torus actions on Riemannian manifolds. We apply our results to the classification problem of  compact, simply connected Riemannian $4$- and $5$-manifolds with positive or nonnegative sectional curvature.
\end{abstract}

\maketitle

\tableofcontents

%-----------------------------------------------------------------------------
% END FRONT MATTER ----------------------------------------------------------
%-----------------------------------------------------------------------------

%-----------------------------------------------------------------------------
%	MAIN MATTER	-----------------------------------------------------------------------------
%-----------------------------------------------------------------------------

%---------------------------------------------
% SECTION: INTRODUCTION
%---------------------------------------------

\section{Introduction}
The study of effective smooth torus actions on  compact, smooth manifolds has a rich and long tradition in the theory of  smooth transformation groups (cf.~\cite{Br, Ko}). 
In Riemannian geometry, starting with Hsiang and Kleiner's topological classification of compact Riemannian  $4$-manifolds of positive (sectional) curvature with an effective isometric  circle action \cite{HK}, isometric actions of tori have been successfully used to obtain classification results on compact Riemannian manifolds with positive or nonnegative curvature and large isometry groups (cf.~\cite{Gr, Gr2, Wi2, Zi}).

The present paper's main contribution is the observation that several results on smooth torus actions on compact smooth simply connected manifolds, and on isometric torus actions on compact simply connected Riemannian manifolds with positive or nonnegative curvature, hold under less restrictive conditions which do not involve the existence of a group action. Indeed, many of these results do not hold because of the presence of a torus action, but rather because the orbit decomposition of the manifold has the structure of a \emph{singular Riemannian foliation} whose leaves are diffeomorphic to flat tori of possibly different dimensions. To make this statement precise, we introduce a special class of singular Riemannian foliations, \emph{B-foliations}, which generalize isometric torus actions on complete Riemannian manifolds. Roughly speaking, a B-foliation $(M,\fol)$ is a partition of a complete Riemannian manifold $M$ into connected closed submanifolds, called the \emph{leaves} of $\fol$, all of which are homeomorphic to some flat manifold and are at a constant distance from each other. More generally, B-foliations are a particular instance of singular Riemannian foliations whose leaves are homeomorphic to some closed aspherical manifold. We will call such singular Riemannian foliations \emph{A-foliations}.

The fact that A-foliations are more general than isometric torus actions is clear for several reasons. On the one hand, the leaves need not be tori. On the other hand, even when the leaves of an A-foliation $(M,\fol)$ on a complete Riemannian manifold $M$ are diffeomorphic to standard tori, the foliation may not be homogeneous, i.e.~there might not be a global torus action on $M$ inducing the given singular Riemannian foliation $\fol$.
This occurs, for example, when the distribution of the tangent spaces of the torus leaves is not orientable. Moreover, our results hold when the leaves carry exotic smooth structures, e.g.~in the case of exotic tori. Nevertheless, we do not know of any non-trivial examples of singular Riemannian foliations whose leaves are exotic tori. It would be interesting to find B-foliations by exotic tori on  simply connected manifolds and, in particular, on spheres.

% MAIN RESULTS ON B-FOLIATIONS

In this paper we focus our attention on A- and B-foliations on compact simply connected Riemannian manifolds. Although every aspherical manifold can appear as the regular leaf of  an A-foliation, our first result implies that the simply connected case is considerably more rigid.

% THEOREM A

\begin{theorem}
\label{T:fund_gp}
Let $(M,\fol)$ be a closed singular Riemannian foliation on a compact, simply connected Riemannian manifold $M$. If $L$ is a regular leaf of $\fol$, then $\pi_1(L)$ is isomorphic to $A\times K_2$, where $A$ is abelian and $K_2$ is a finite $2$ step nilpotent $2$-group. 
\end{theorem}

In particular, we have the following corollary:

% COROLLARY B

\begin{corollary}
\label{C:torus_leaves}Let $(M,\fol)$ be an A-foliation on a compact Riemannian manifold $M$. If $M$ is  simply connected, then the regular leaves are homeomorphic to tori.
\end{corollary}

Observe that $2$ step nilpotent $2$-groups already appear as fundamental groups of regular leaves of codimension one singular Riemannian foliations and hence cannot be  avoided in the statement of Theorem~A (cf.~\cite[Table~1.4]{GH}). This occurs, for example, in the cohomogeneity one $\SO(3)$ action on $\sphere^4$.

With Theorem~\ref{T:fund_gp} and Corollary~\ref{C:torus_leaves} in place, we extend to A- and B-foliations several basic results on smooth effective torus actions on smooth compact manifolds (cf.~\cite{Ko,OR, Oh, KMP}). We prove:

% THEOREM C

\begin{theorem}
\label{T:Kobayashi}
Let $(M,\fol)$ be a B-foliation on a compact Riemannian manifold $M$ and let $\Sigma_0\In M$ denote the stratum of $0$-dimensional leaves. Then $\chi(\Sigma_0)=\chi(M)$.
\end{theorem}

% THEOREM D

\begin{theorem}
\label{T:CODIM_1_BFOL_SC}
The only codimension $1$ A-foliations on compact, simply connected Riemannian manifolds are the homogeneous singular Riemannian foliations $(\sphere^2,\sphere^1)$ and  $(\sphere^3,T^2)$.
\end{theorem}

% THEOREM E

\begin{theorem}\label{T:BFOL_CODIM2}
Let $(M^{n+2},\fol^n)$ be a codimension $2$ A-foliation on a compact, simply connected Riemannian manifold $M^{n+2}$ with $n\geq 1$. Then either $M=\sphere^3$ and $\fol$ is given by a weighted Hopf action, or the following hold:
\\
\begin{enumerate}

\item The leaf space $B=M/\fol$ is homeomorphic to a $2$-disk, the interior of $B$ is smooth, and the boundary $\partial B$ consists of at least $n$ totally geodesic segments meeting at an angle of $\pi/2$. \\

\item  Let $L_0$ be a generic leaf and $L_1$ a singular leaf. Then there is a submersion $L_0\to L_1$, with fiber $\mathbb{S}^1$ if $L_1$ belongs to a geodesic in $\partial B$, or with fiber $T^2$, if $L_1$ belongs to a vertex of $\partial B$.
\end{enumerate}
\end{theorem}

% COMMENTS ON F-STRUCTURES

Another generalization of isometric torus actions are the so-called  \emph{F-structures}, introduced by Cheeger and Gromov \cite{ChGr1,ChGr2}. These structures are, roughly speaking,  generalized local torus actions and play a central role in the Cheeger-Fukaya-Gromov theory of collapsed Riemannian manifolds with bounded sectional curvature (cf.~\cite{CFG,Fu}). The so-called  \emph{pure} F-structures (see \cite{ChGr1}) give rise to B-foliations with leaves diffeomorphic to flat manifolds. Recall that, by work of Cheeger, Fukaya and Gromov, there exists a constant $\epsilon(n, d) > 0$, such that any compact Riemannian manifold $M^n$ with curvature $|\sec(M^n)| \leq  1$, $\diam(M^n) < d$ and $\mathrm{vol}(M^n) < \epsilon(n, d)$ admits a pure F-structure (see \cite{ChRo} and references therein).
 Therefore, $M^n$ is B-foliated. 

Although B-foliations resemble F-structures, the two  concepts are independent. A B-foliation on a Riemannian manifold does not necessarily correspond to an F-structure since, for instance, B-foliations with exotic torus leaves cannot be generated by F-structures. On the other hand, certain F-structures among those that are not pure do not generate a B-foliation. 
\\

% GEOMETRIC APPLICATIONS

As an application of our results, we extend work in \cite{HK,Kl,SY,GGS2} on positively and nonnegatively curved compact, simply connected Riemannian manifolds with large effective isometric to\-rus actions  to the case of A-foliations. Recall that a Riemannian manifold $(M,g)$ is said to have \emph{quasi-positive curvature}  if $(M,g)$ has nonnegative (sectional) curvature and a point with strictly positive curvature. 

% THEOREM F

\begin{theorem}\label{T:QPC_D4_CODIM2}
Let $(M^n,g)$ be a compact, simply connected Riemannian $n$-manifold with quasi-positive curvature supporting a codimension $2$ A-folia\-tion.
\\
\begin{enumerate}
	\item If $n=4$, then  $M^4$ is diffeomorphic to $\sphere^4$ or $\CP^2$.\\
	\item If $n=5$, then  $M^5$ is diffeomorphic to $\sphere^5$.
\end{enumerate}
\end{theorem}

% THEOREM G

\begin{theorem}
\label{T:NNC_D4_CODIM2}
Let $(M^n,g)$ be a compact,  simply connected Riemannian $n$-manifold with nonnegative curvature and a codimension $2$ A-foliation.
\\
\begin{enumerate}
	\item If $n=4$, then  $M^4$ is diffeomorphic to $\sphere^4$, $\CP^2$, $\CP^2\#\pm\CP^2$ or $\sphere^2\times\sphere^2$.\\
	\item If $n=5$, then $M^5$ is diffeomorphic to $\sphere^5$ or to one of the two $\sphere^3$-bundles over $\sphere^2$.
\end{enumerate}
\end{theorem}

% THEOREM H

\begin{theorem}\label{T:D4_S1FOL}
Let $M$ be a  compact, simply connected Riemannian $4$-manifold with a singular Riemannian foliation by circles. Then the foliation is induced by a smooth circle action and the following hold:
\\
\begin{enumerate}
	\item If $M$ has positive curvature, then $M$ is diffeomorphic to $\sphere^4$ or $\CP^2$.\\
	\item If $M$ has nonnegative curvature, then $M$ is diffeomorphic to $\sphere^4$, $\CP^2$, $\CP^2\#\pm\CP^2$ or $\sphere^2\times\sphere^2$.\\
\end{enumerate}
\end{theorem}

We call a B-foliation \emph{Euclidean} if its regular leaves are flat with the induced Riemannian metric and define the \emph{Euclidean rank} of a Riemannian manifold $(M,g)$ as the maximum dimension of Euclidean B-foliations on $M$ compatible with the fixed metric $g$. This invariant generalizes the symmetry rank of $(M,g)$ (cf.~\cite{GS01}). It follows from Otsuki's lemma \cite[Lemma~3.3, p. 224]{dC} (cf.~also an argument due to Wilking found in \cite{GG}) that the Euclidean rank of a compact, quasi-positively curved Riemannian $n$-manifold is less than or equal to $\lfloor (n+1)/2\rfloor$. In dimension $n\leq 9$, it is easy to show, following the comparison arguments in the proof of Theorem~\ref{T:NNC_D4_CODIM2}, that the Euclidean rank of a compact, simply connected Riemannian $n$-manifold of nonnegative curvature is bounded above by $\lfloor 2n/3 \rfloor$.
These bounds coincide with the corresponding bounds for the symmetry rank (cf.~\cite{GS01,GGS2}).

\bigskip
Our paper is structured as follows. In Section~\ref{basic} we recall some basic facts on singular Riemannian foliations. In Section~\ref{S:B_fol} we introduce  A- and B-foliations and show that the infinitesimal foliation at any point of a manifold with an A- or a B-foliation is also an A- or a B-foliation, respectively. 
Sections~\ref{S:FUNDGP_LEAVES} through \ref{S:CODIM_2} contain the proofs of Theorems~\ref{T:fund_gp} through \ref{T:BFOL_CODIM2}.
Section~\ref{S:CURV_1} contains the proof of  Theorems~\ref{T:QPC_D4_CODIM2} and~\ref{T:NNC_D4_CODIM2}. 
Finally, in Section~\ref{S:CURV_2} we prove Theorem~\ref{T:D4_S1FOL}.
% ADDED
Throughout our paper we will assume all manifolds to be connected and without boundary, unless stated otherwise.

% ACKNOWLEDGEMENTS

\begin{ack}
The second named author thanks Xiaoyang Chen for some initial conversations. Both authors wish to acknowledge the hospitality of the  Mathematisches Forschungsinstitut Oberwolfach, where the work presented in this paper was initiated. The authors also thank Wolfgang Ziller for comments on a first version of this paper, Alexander Lytchak for suggestions which led to the proofs of Theorems~\ref{T:fund_gp} and \ref{T:fibration-flat}, and the referee for suggesting improvements to our original results.  
\end{ack}

%------------------------------------------------
%	SECTION: PRELIMINARIES
%------------------------------------------------

\section{Preliminaries}
\label{basic}

In this section we collect some background material  on singular Riemannian foliations. We refer the reader to \cite{Mo,ABT} for further results on the theory.

%	SSUBSECTION: SINGULAR RIEMANNIAN FOLIATIONS

\subsection{Singular Riemannian foliations}
\label{SS:SRF} A \emph{transnormal system} $\mathcal{F}$ on a complete Riemannian manifold $M$ is a decomposition of $M$ into smooth, complete, injectively immersed connected submanifolds, called \emph{leaves}, such that every geodesic emanating perpendicularly to one leaf remains perpendicular to all leaves. A transnormal system $\mathcal{F}$ is called a \emph{singular Riemannian foliation}  if there are smooth vector fields $X_i$ on $M$ such that, for each point $p\in M$, the tangent space $T_p L_p$ to the leaf $L_p$ through $p$ is given as the span of the vectors $X_i(p)\in T_pM$. We will call the quotient space $M/\mathcal{F}$ the \emph{leaf space}, and will also denote it by $M^*$. We will let $\pi: M\to M/\fol$ be the leaf projection map. The pair $(M,\fol)$ will denote a singular Riemannian foliation $\fol$ on a complete Riemannian manifold $M$. Slightly abusing notation, we will also refer to the pair $(M,\fol)$ as a singular Riemannian foliation.  

A singular Riemannian foliation $\fol$ will be called \emph{closed} if all its leaves are closed in $M$; the foliation will be called \emph{locally closed} at $x\in M$ if, for some neighborhood $U$ of $x$, the restriction of $\fol$ to $U$ is closed, i.e.~connected components of the intersection of the leaves of $\fol$ with $U$ are closed in $U$. If $\fol$ is locally closed at $x$, then the local quotient $U/\fol$ is a well defined Alexandrov space of curvature locally bounded from below. Similarly, if $\fol$ is closed, then the quotient space $M/\fol$ is an Alexandrov space of curvature locally bounded below.  
% REFEREE SUGGESTiON
We will henceforth only consider closed singular Riemannian foliations

%	SSUBSECTION:  NOTATION GROUP ACTIONS

\subsection{Group actions}
 As group actions will appear throughout our work, let us fix some notation before proceeding. Given a Lie group $G$ acting (on the left) on a smooth manifold $M$,   we denote by $G_p=\{\, g\in G : gp=p\, \}$ the \emph{isotropy group} at $p\in M$ and by $Gp=\{\, gp : g\in G\, \}\simeq G/G_p$ the \emph{orbit} of $p$. The \emph{ineffective kernel} of the action is the subgroup $K=\cap_{p\in M}G_p$. We say that $G$ acts \emph{effectively} on $M$ if $K$ is trivial. The action is  \emph{free} if every isotropy group is trivial. Given a subset $A\subset M$, we will denote its image in $M/G$ under the orbit projection map by $A^*$. When convenient, we will also denote the orbit space $M/G$ by $M^*$.

% EXAMPLE: GROUP ACTION

% ADDED COMPLETE TO THE MANIFOLD
\begin{example}[Isometric Lie group actions] 
Perhaps the most familiar example of a singular Riemannian foliation is the one induced by an (effective) isometric action of a Lie group  $G$ on a complete Riemannian manifold $M$. In this case, the foliation is given by the orbits of the action, and we say that the foliation is a \emph{homogeneous foliation}. If $G$ is compact, then the foliation is closed, and it is locally closed if and only if all the slice representations $G_p\to O(\nu_p(G p))$ have compact image. 
\end{example}

% REMARK

\begin{rem}  
We will sometimes denote a  homogeneous foliation, given by the action of a Lie group $G$, by $(M,G)$, provided the $G$-action is understood.
\end{rem}

%	SSUBSECTION: STRATIFICATION

\subsection{Stratification}
Let $M$ be a complete Riemannian manifold with a closed singular Riemannian foliation $\mathcal{F}$. The \emph{dimension} of $\mathcal{F}$, denoted by $\dim \mathcal{F}$, is the maximal dimension of its leaves. The \emph{codimension} of $\mathcal{F}$ is, by definition, 
\[
\codim(\mathcal{F},M)=\dim M - \dim \mathcal{F}.
\]
For $k\leq \dim \mathcal{F}$, define 
\[
\Sigma_{(k)} =\{\, p\in M : \dim L_p =k\,\}.
\]
Every connected component $C$ of the set $\Sigma_{(k)}$ is an embedded (possibly non complete) submanifold of $M$ and the restriction of $\mathcal{F}$ to $C$ is a Riemannian foliation. Given $p\in M$, let $\Sigma^{p}$ be the connected component of $\Sigma_{(k)}$ through $p$, where $k=\dim L_p$. We will refer to the decomposition of $M$ into the submanifolds $\Sigma^{p}$ as the \emph{canonical stratification} of $M$.

The subset $\Sigma_{(\dim \mathcal{F})}$ is open, dense and connected in $M$; it is called the \emph{regular stratum} of $M$. It will be denoted by $M_0$ and its points will be called \emph{regular points}. If $M_0=M$ we say that the foliation is \emph{regular}. All other strata $\Sigma^p$ have codimension at least $2$ in $M$ and are called \emph{singular strata}. For any singular stratum $\Sigma^{p}$, we have 
\[
\codim(\mathcal{F},\Sigma^{p})<\codim (\mathcal{F},M).
\]

%	SSUBSECTION: INFINITESIMAL SINGULAR RIEMANNIAN FOLIATIONS

\subsection{Infinitesimal singular Riemannian foliations}
\label{SSS:INF_SRF} Let $M$ be a complete Riemannian manifold with a closed singular Riemannian foliation $\mathcal{F}$. Given a point $p\in M$ and some small $\epsilon>0$, let $S_p=\exp_p (\nu_pL_p) \cap B_{\epsilon}(p)$ be a \emph{slice} through $p$, where $B_{\epsilon}(p)$ is the distance ball of radius $\epsilon$ around $p$. The foliation $\mathcal{F}$ induces a foliation $\mathcal{F}|_{S_p}$ on $S_p$ by letting the leaves of $\mathcal{F}|_{S_p}$ be the connected components of the intersection between $S_p$ and the leaves of $\mathcal{F}$. The foliation $(S_p,\mathcal{F}|_{S_p})$ may not be a singular Riemannian foliation with respect to the induced metric on $S_p$. Nevertheless, the \emph{pull-back} foliation $\exp_p^{*}(\mathcal{F})$ is a singular Riemannian foliation on $\nu_pL_p\cap B_{\epsilon}(0)$ equipped with the Euclidean metric (cf.~\cite[Proposition 6.5]{Mo}), and it is invariant under homotheties fixing the origin (cf.~\cite[Lemma 6.2]{Mo}). In particular, it is possible to extend $\exp^*(\mathcal{F})$ to all of $\nu_pL_p$, giving rise to a singular Riemannian foliation $(\nu_pL_p,\mathcal{F}_p)$ called the \emph{infinitesimal foliation} of $\mathcal{F}$ at $p$.

Notice that $0\in \nu_p(L_p)$ is always a leaf of the infinitesimal foliation $\mathcal{F}_p$. By definition, leaves stay at a constant distance from each other, in particular every leaf stays in some distance sphere around the origin, and it makes sense to consider the infinitesimal foliation restricted to the unit sphere. Since the infinitesimal foliation is invariant under homothetic transformations, it can be reconstructed from its own restriction to the unit sphere. Taking this into account, we will sometimes refer to $(\nu^1_pL_p,\mathcal{F}_p)$ also as the infinitesimal foliation at $p$ and shall write $(\sphere^\perp_{p}, \fol_{p})$.

Given two points $p_1, p_2$ in some leaf $L$, the corresponding infinitesimal foliations $(\sphere^\perp_{p_1}, \fol_{p_1})$, $(\sphere^\perp_{p_2}, \fol_{p_2})$ are foliated isometric, in the sense that there is a (non-canonical) linear isometry $\sphere^\perp_{p_1}\to \sphere^\perp_{p_2}$ preserving the foliation. Moreover, these foliations can be  glued together to give a foliation on $\nu^1(L)$ in the following sense: If one identifies $\nu^1(L)$  via the normal exponential map with $\partial \mathrm{Tub}_{\epsilon}(L)$, the boundary of an $\epsilon$-tubular neighborhood of $L$, then the intersections of leaves in $\fol$ with with $\sphere^\perp_p$ are exactly the leaves in $\fol_p$. In particular, if $L'$ is a leaf in $\partial \mathrm{Tub}_{\epsilon}(L)\simeq \nu^1L$, then $L'$ is a union of infinitesimal leaves. Moreover, if $p\in L$ and $q\in L'$ can be written as $q=\exp_p \epsilon v$, $v\in \sphere^\perp_p$, then the connected components of a fiber of $p$ under the metric projection $L'\to L$ (which is a submersion, cf. \cite[Lemma~6.1]{Mo}) are given by $\mc{L}_v$, where $\mc{L}_v\in \fol_p$ are diffeomorphic to the infinitesimal leaf of $\sphere^\perp_p$ passing through $v$. Therefore, there is a fibration
\begin{equation}
\label{E:Inf_fibration}
\mc{L}_v\to L_{q}\to \bar{L}_p
\end{equation}
for some finite cover $\bar{L}_p\to L_p$.

% REMARK: SPACE OF DIRECTIONS

\begin{rem}
\label{R:Space_Dir} 

Let $(M,\fol)$ be a closed singular Riemannian foliation. As recalled in Section~\ref{SS:SRF}, the leaf space $M^*$ is an Alexandrov space of curvature locally bounded below. Let us quickly recall the procedure to compute the space of directions $\Sigma_{p^*}$ at a point $p^*\in M^*$. Let $p\in M$ be a point in the preimage of $p^*$ and let $\sphere_{p^*}=\sphere^\perp_p/\fol_p$ be the quotient of the infinitesimal foliation $(\sphere^\perp_p,\fol_p)$. The fundamental group $\pi_1(L_p)$ acts on $\sphere_{p^*}$ by isometries via the so-called \emph{holonomy action} and $\Sigma_{p^*}$ is isometric to $\sphere_{p^*}/\pi_1(L_p)$. Given $v\in \sphere^\perp_p$ with image  $v^*\in \sphere_{p^*}$, let $H$ be the subgroup of $\pi_1(L_p)$ fixing $v^*$. Then, in fibration~\eqref{E:Inf_fibration}, the cover $\bar{L}_p$ is $\tilde{L}_p/H$, where $\tilde{L}_p$ is the universal cover of $L_p$. 

\end{rem}

% EXAMPLE

\begin{example}
Let $(M,G)$ be a homogeneous foliation.  Given a point $p\in M$, the connected component $G_p^0$ of the isotropy group $G_p$ acts on $\nu_p(G p)$ by isometries, via the so-called \emph{slice representation}. In this case, the infinitesimal foliation $\mathcal{F}_p$ is the homogeneous foliation given by the orbits of $G_p^0$ on $\nu_p(G p)$. Given $q\in M$ close to $p$, with isotropy $G_q<G_p$, the projection \eqref{E:Inf_fibration} is the projection 
\[
G_p^0/G_q\to G/G_q \to G/G_p^0,
\]
where $G/G_p^0$ is a cover of the orbit $G/G_p$ though $p$. 
\end{example}

%	SSUBSECTION: THE MOLINO CONSTRUCTION

\subsection{The Molino bundle}\label{SS:Molino}

We conclude this section by recalling the main properties of the so-called \emph{Molino bundle} (cf.~\cite[Proposition 4.1]{Mo}). We let $(M,\fol)$ be a closed singular Riemannian foliation of codimension $q$ on a compact Riemannian manifold $M$. Since the foliation on the regular stratum $M_0$ is regular, there exists a principal $\O(q)$-bundle $\hat{M}\to M_0$, called the \emph{Molino bundle}, and a foliation $(\hat{M},\hat{\fol})$ such that the leaves of $\hat{\fol}$ are Galois covers of the leaves of $\fol$. Moreover, the leaves of $\hat{\fol}$ are actually diffeomorphic to those of $\fol$ on an open dense set. In addition, since $\fol$ is closed, the leaves of $\hat{\fol}$  are given by the fibers of a submersion $\theta:\hat{M}\to W$, where $W$ is the frame bundle of the orbifold $M_0/\fol$. In particular, $W$ is a manifold with an almost free smooth $\O(q)$-action and $\theta$ is $\O(q)$-equivariant.

Let $\hat{M}_{\O(q)}=\hat{M}\times_{\O(q)}E\O(q)$ and $W_{\O(q)}=W\times_{\O(q)}E\O(q)$ be the Borel constructions of $\hat{M}$ and $W$, respectively. Then $\hat{M}_{\O(q)}$ is homotopy equivalent to $M_0$  and $\theta$ induces a fibration $\hat{\theta}: \hat{M}_{\O(q)}\to W_{\O(q)}$ with the same fibers as $\theta:\hat{M}\to W$. Furthermore, the space $B=W_{\O(q)}$ coincides with Haefliger's classifying space of the orbifold $M_0/\fol$ (cf.~\cite{Hf}). Therefore, up to homotopy, there is a fibration
\begin{equation}\label{E:Borel_Molino}
L\stackrel{\iota}{\lra} M_0\stackrel{\hat{\theta}}{\lra} B,
\end{equation}
where $L$ is a regular leaf of $\fol$.
\\

%------------------------------------------------------------------
%	SECTION: BIEBERBACH FOLIATIONS
%------------------------------------------------------------------

\section{A-foliations and B-foliations}
\label{S:B_fol}
We now introduce \emph{A-foliations} and \emph{B-foliations}, which are the main object of study in our paper. 

\begin{defn}[A-foliation] 
A closed singular Riemannian foliation $(M,\fol)$ is an \emph{A-foliation} if every leaf is an aspherical manifold.
\end{defn}

% DEFINITION

\begin{defn}[B-foliation]
A closed singular Riemannian foliation $(M,\fol)$ is a \emph{B-foliation} if every leaf is homeomorphic to some Bieberbach manifold.
\end{defn}

Recall that  a \emph{Bieberbach manifold} is a manifold diffeomorphic to $\RR^n/G$, where $G$ is a discrete group of Euclidean isometries acting freely and cocompactly on $\RR^n$. These groups are called \emph{Bieberbach groups}. Abstractly, Bieberbach groups can be characterized as torsion-free groups with a normal finite index abelian subgroup (cf.~\cite{Wo}). In particular, every subgroup of a Bieberbach group is a Bieberbach group. Every Bieberbach manifold is compact, has no boundary and admits a flat Riemannian metric. 

% REMARK

\begin{rem}
\label{R:Borel_Conj}
We shall use the fact that any aspherical manifold $N$ with fundamental group isomorphic to a Bieberbach group $G$ must be homeomorphic to a Bieberbach manifold. To see this, observe first that $N$ must be homotopy equivalent to $\RR^n/G$, since both are models for $K(G,1)$. It then follows from the solution of the Borel conjecture for flat manifolds (cf.~\cite[Section 4]{Fa} and \cite{BBBP}) that $N$ and $\RR/G$ must be homeomorphic. The manifolds  $N$ and $\RR^n/G$ may not be diffeomorphic, as illustrated by the existence of exotic  tori, which appear already in dimension $5$  (cf.~\cite{HS}). 
\end{rem}

% REMARK

\begin{rem}
\label{R:Borel_conj}
In the preceding definition we do not assume that the leaves are flat with the induced Riemannian metric.
\end{rem}

% EXAMPLE

\begin{example} Every isometric torus action on a complete Riemannian manifold induces a (homogeneous) B-foliation.  Bundles whose fibers are homeomorphic to Bieberbach manifolds are also examples of B-foliations. 
\end{example}

% EXAMPLE

\begin{example}[Non-homogeneous B-foliations] The simplest way to construct non-homogeneous B-foliations with regular leaves homeomorphic to tori is to take the Riemannian product of a complete Riemannian manifold and an exotic torus. As the leaves are exotic tori, they cannot correspond to the orbits of an isometric torus action. 

One can also construct non-homogeneous B-foliations whose leaves are diffeomorphic to flat tori in the following way. Let $B$ be a compact smooth manifold with non-trivial fundamental group, let $T^n$ be a standard $n$-di\-men\-sional torus, and $\rho:\pi_1(B)\to \Diff(T^n)$ a homomorphism. Let $\tilde{B}$ be the universal cover of $B$ and let $\pi_1(B)$ act diagonally on the product $\tilde{B}\times T^n$. This action is free and, taking the quotient, we obtain a fiber bundle 
\[
T^n\rightarrow \tilde{B}\times_{\pi_1(B)} T^n \rightarrow B.
\]
The total space $ \tilde{B}\times_{\pi_1(B)} T^n$ is B-foliated by the fibers of the bundle. If the B-foliation is homogeneous, then the bundle is principal, and the structure group reduces to a subgroup of $T^n\In \Diff(T^n)$. In particular, the structure group is contained in the identity component  $\Diff_0(T^n)$ of $\Diff(T^n)$. Thus, to construct a non-homogeneous B-foliation, it is enough to consider a homomorphism $\rho:\pi_1(B)\to \Diff(T^n)$ whose image is not entirely contained in $\Diff_0(T^n)$.

As a concrete example, if $B=\sphere^1$, $T^n=\sphere^1$, and $\rho: \pi_1(B)=\ZZ\to \Diff(\sphere^1)\simeq \O(2)$ is an orientation reversing diffeomorphism, then $ \tilde{B}\times_{\pi_1(B)}T^n$ is a Klein bottle $K$, and the foliation is given by the fibers of the submersion $K\to \sphere^1$.
\end{example}

In the case of A- and B-foliations, the total space and the base in fibration~\eqref{E:Inf_fibration} are homeomorphic, respectively,  to aspherical and Bieberbach manifolds. The following general result shows that these classes of manifolds are well-behaved with respect to fibrations. 

% THEOREM: FIBRATION BY BIEBERBACH MANIFOLDS

\begin{thm}\label{T:fibration-flat}
Let $F$, $M$, and $N$ be topological manifolds and let $F\to M\to N$ be a fibration.
\begin{enumerate}
	\item  If $M$ is aspherical, then $F$ and $N$ are aspherical.
	\item  If $M$ is homeomorphic to a Bieberbach manifold, then $F$ and $N$ are homeomorphic to Bieberbach manifolds.
\end{enumerate}
\end{thm}

% PROOF

\begin{proof}

We first prove part (1). Consider the fibration between the universal covers $ \tilde{M}\to \tilde{N}$, with fiber $H$.  Since $\tilde{M}$  is contractible and $\tilde{N}$ is simply connected, we can apply the Serre spectral sequence with integral coefficients, and from it we obtain that $H$ and $\tilde{N}$ are contractible. In fact, if $H^*(H)$ has cohomological dimension $a$, and $H^*(\tilde{N})$ has cohomological dimension $b$, then $H^*(\tilde{M})$ has cohomological dimension $a+b$ and this has to be $0$. In particular, $\tilde{N}$ is aspherical. Therefore, $N$ is aspherical and, from the long exact sequence in homotopy, so is $F$. 

Now we prove part (2). By Remark~\ref{R:Borel_Conj} it suffices to show that $F$ and $N$ are aspherical and $\pi_1(F)$ and $\pi_1(N)$ are Bieberbach groups. Since $M$ is homeomorphic to a Bieberbach manifold, it follows from part (1) that $F$ and $N$ are aspherical. From the long exact sequence of the fibration, we have
\[
1\to\pi_1(F)\to\pi_1(M)\to \pi_1(N)\to 1,
\]
where $\pi_1(M)$ is a Bieberbach group, i.e.~a torsion free group with a finite index normal abelian subgroup. Since $\pi_1(F)$ is a subgroup of a Bieberbach group, it is again a Bieberbach group and therefore $F$ is homeomorphic to a flat manifold.

We now prove that $\pi_1(N)$ is a Bieberbach group. First, we show that $\pi_1(N)$ is torsion free. Suppose this is not the case. Then there is a finite cyclic subgroup $\ZZ_k$ acting freely on the contractible manifold $\tilde{N}$. It follows that $\tilde{N}/\ZZ_k$ is a $K(\ZZ_k,1)$, which contradicts the fact that $K(\ZZ_k,1)$ has infinite cohomological dimension.

Finally, let us show that $\pi_1(N)$ contains a finite index normal abelian subgroup.  Since $\pi_1(M)$ is a Bieberbach group, there exists a finite index normal subgroup $\ZZ^d\In \pi_1(M)$. The image of $\ZZ^d$ in $\pi_1(N)$ is a finitely generated normal torsion free abelian group $A$. Since the map $\pi_1(M)/\ZZ^d\to \pi_1(N)/A$ is surjective, $A$ has finite index in $\pi_1(N)$. Therefore $\pi_1(N)$ is a Bieberbach group, and therefore $N$ is homeomorphic to a Bieberbach manifold.
\end{proof}

Theorem  \ref{T:fibration-flat} can be applied to fibration \eqref{E:Inf_fibration} to obtain the following corollaries.

% COROLLARY

\begin{cor} Let  $(M,\fol)$ be a closed singular Riemannian foliation.
\begin{enumerate}
	\item  If a regular leaf of $(M,\fol)$ is aspherical, then $(M,\fol)$ is an A-foliation.
	\item  If a regular leaf of $(M,\fol)$ is homeomorphic to a flat manifold, then $(M,\fol)$ is a B-foliation.
\end{enumerate}
\end{cor}

% COROLLARY

\begin{cor} \hfill

\label{C:Inf_B_fol}
\begin{enumerate}
	\item The infinitesimal foliations of an A-foliation are again A-foliations.
	\item The infinitesimal foliations of a B-foliation are again B-foliations.
\end{enumerate}
\end{cor}

% HOMOGENEITY OF CIRCLE FOLIATIONS

Since B-foliations generalize torus actions, it is natural to ask when such a foliation is homogeneous. We will now show that closed singular Riemannian foliations by circles on compact, simply connected Riemannian manifolds are homogeneous, answering the simplest instance of this question.  We first prove the following general lemma. 

% LEM

\begin{lem}
\label{L:ORIENT_FOL}
Let $(M,\fol)$ be a closed, singular Riemannian foliation on a compact, simply connected Riemannian manifold $M$. Then the foliation $(M,\fol)$ restricted to the regular part $M_0$ is orientable.
\end{lem}

\begin{proof} By \cite[Proposition~3.7]{Mo}, the leaf space $M_0/\fol$ is an orbifold. Lytchak showed in \cite[Corollary~5.3]{Ly10} that $\pi^{\mathrm{orb}}_1(M_0/\fol)=1$, i.e.~the classifying space $B$ of $M_0/\fol$ is simply connected. In particular, $B$ is orientable, which implies the result.
\end{proof}

% THM: HOMOGENEITY

\begin{thm}
\label{T:HOMOGENEITY} Let $(M,\fol)$ be a closed, singular Riemannian foliation on a compact, simply connected Riemannian manifold. If the regular leaves of the foliation are circles, then the foliation is homogeneous.
\end{thm}

\begin{proof} By Lemma~\ref{L:ORIENT_FOL}, the foliation $\fol$ restricted to the regular part $M_0$ is orientable. Hence $(M_0,\fol)$ is given by a circle action. Since the singular strata of the foliation are smooth closed embedded submanifolds, the action can be extended to the singular strata by radially extending it on small tubular neighborhoods around each component of the singular strata. 
\end{proof}

\begin{rem}
In the subsequent sections we will  assume  all manifolds to be compact, unless stated otherwise.
\end{rem}

%---------------------------------------------------------------------------------------------------------------
%	SECTION: FUNDAMENTAL GROUPS OF LEAVES (PROOF OF THEOREM A)
%---------------------------------------------------------------------------------------------------------------

\section{The fundamental group of a regular leaf}
\label{S:FUNDGP_LEAVES}

\subsection*{Proof of Theorem~\ref{T:fund_gp}}
Let $(M,\fol)$ be a closed singular Riemannian foliation on a compact, simply connected Riemannian manifold. Observe that $M$ remains simply connected if we discard all the singular strata of codimension at least $3$.
 Therefore, we may assume that 
\[
M=M_0\cup \bigcup_{i=1}^r\Sigma_i,
\]
where $\Sigma_i$ are the connected components of codimension $2$ of the singular stratum and $M_0$ is the regular stratum of $\fol$. We will assume that $\cup_{i=1}^r\Sigma_i$ is the empty set if there are no codimension $2$ strata.
\\

For each $i=1,\ldots r$, let $U_i$ be a small tubular neighborhood of $\Sigma_i$ with foot-point projection $\ol{\P}_i:U_i\to \Sigma_i$.  The restriction of $\ol{\P}_i$ to $U_i\setminus \Sigma_i$ is a circle bundle. The fibers of this circle bundle define a free homotopy class $[c_i]$ of loops in $M_0$. Moreover, if $p_i\in U_i\setminus\Sigma_i$, then $L_{p_i}$ is entirely contained in $U_i$ and the restriction of $\ol{\P}_i$ to $L_{p_i}$ is a circle bundle
\begin{equation}\label{E:projection}
\sphere^1\to L_{p_i}\to L_{\ol{\P}_i(p_i)}.
\end{equation}

Let us fix a regular leaf $L_0$ in $M_0$ and a point $p_0\in L_0$. For $i=1\ldots r$, fix a horizontal curve $\gamma_i:[0,1]\to M_0$ from $p_0$ to some $p_i\in U_i\setminus\Sigma_i$. This curve induces an isomorphism $h_i:L_{p_i}\to L_0$, given by $h_i( p)=\gamma_p(0)$, where $\gamma_p$ is the only horizontal curve ending at $p$ whose projection to $M_0/\fol$ coincides with the projection of $\gamma_i$. Let $c_i$ be a representative of the free homotopy class $[c_i]$ defined in the preceding paragraph, passing through $p_i$. The element $k_i={h_i}_*(c_i)\in \pi_1(L_0,p_0)$ is uniquely determined, up to a sign, by $\gamma_i$ and $c_i$. Let $K \In \pi_1(L_0,p_0)$ be the group generated by the elements $k_i$. % K IS NORMAL 
Notice that each group $\langle k_i\rangle$  generated by $k_i$ is normal in $\pi_1(L_0,p_0)$ and therefore $K$, being generated by normal subgroups, is normal in $\pi_1(L_0,p_0)$ as well.

If $\gamma_i':[0,1]\to M_0$ is a second horizontal curve from $p_0$ to $p_i'\in U_i\setminus \Sigma_i$, a different homeomorphism $h_i':L_{p_i'}\to L_0$ is induced, and we obtain a different element $k_i'\in \pi_1(L_0,p_0)$.  Letting $\iota:L_0\to M_0$ denote the inclusion of $L_0$ in $M_0$, the elements $\iota_*(k_i)$ and $\iota_*(k_i')$ in $\pi_1(M_0,p_0)$ are conjugate by an element of $\pi_1(M_0,p_0)$. 
\\

Recall from Section \ref{SS:Molino} that, up to homotopy, there is a fibration
\begin{equation}
\label{E:homotopy_fibration}
	L_0\stackrel{\iota}{\lra} M_0 \stackrel{\hat{\theta}}{\lra} B,
\end{equation}
where $\iota: L_0\to M_0$ is the inclusion and $B$ is Haefliger's classifying space of the orbifold $M_0/\fol$. In particular, by definition, $\pi_i^{\mathrm{orb}}(M_0/\fol)=\pi_i(B)$ for all $i>0$.
Let $H$ be the image of the boundary map $\partial:\pi_2(B, b_0)\to \pi_1(L_0,p_0)$ in the homotopy exact sequence of fibration \eqref{E:homotopy_fibration}. There is an exact sequence
\begin{equation}
\label{E:Seq_rigid}
0\to  H\stackrel{\partial}{\lra} \pi_1(L_0,p_0)\stackrel{\iota_*}{\lra} \pi_1(M_0,p_0)\stackrel{\hat{\theta}_*}{\lra}\pi_1(B, b_0)\to 1.
\end{equation}
To prove Theorem~\ref{T:fund_gp}, we will proceed in three steps:
\begin{itemize}
\vspace{.3cm}
	\item[\emph{Step 1:}] $K\In\pi_1(L_0,p_0)$ maps surjectively onto $\pi_1(M_0,p_0)$ under $\iota_*$.\vspace{.3cm}
	
	\item[\emph{Step 2:}] $K$ splits as a product $K_1\times K_2$, where $K_1$ is abelian and $K_2$ is a finite $2$ step nilpotent $2$-group. \vspace{.3cm}
	\item[\emph{Step 3:}] $H\In \pi_1(L)$ is central.
\vspace{.3cm}
\end{itemize}
By the first step, $\pi_1(L_0,p_0)$ is generated by $H$ and $K$. By the second step, $K$ splits as a product $K_1\times K_2$, and by the third step $[K,H]=\{e\}$. Let $A$ denote the group generated by $K_1$ and $H$. Then $A$ and $K_2$ generate $\pi_1(L_0,p_0)$, $[A, K_2]=0$ and $A\cap K_2=\{e\}$. Therefore $\pi_1(L_{0},p_0)$ splits as $A\times K_2$.
\\

% PROOF OF STEP 1

\noindent \textit{Proof of Step 1.} 
Since $\pi_1^{\mathrm{orb}}(M_0/\fol)=\pi_1(B)=1$ by \cite[Corollary 5.3]{Ly10}, the exact sequence in homotopy of fibration \eqref{E:Borel_Molino} implies that $\iota_*$ is surjective. Since $K$ is normal in $\pi_1(L_0,p_0)$, the group $\iota_*(K)$ is normal in $\pi_1(M_0,p_0)$.

Let $c$ be a loop in $M_0$, representing an element $[c]\in \pi_1(M_0,p_0)$. Since $M$ is simply connected, there exists a disk $D\In M$ bounding $c$. We can choose $D$ so that it intersects the strata $\Sigma_i$ transversally, in a finite number of points. For each point $q_{\alpha}$ we can produce a curve $k'_{i_{\alpha}}$ in $D$ going around $q_{\alpha}$ only. The curve $c$ is homotopic to the product of these $k'_{i_{\alpha}}$ and, by the discussion at the beginning of the proof, every such $k'_{i_{\alpha}}$ is conjugate in $\pi_1(M_0,p_0)$ to some $\iota_*(k_i)$ or $\iota_*(k_i)^{-1}$ of $\iota_*(K)$. Since $\iota_*(K)$ is normal, $k'_{i_{\alpha}}$ is an element of $\iota_*(K)$, and so is $[c] $.
\\

% PROOF OF STEP 2

\noindent \textit{Proof of Step 2.} Suppose that $k_i, k_j\in K$ do not commute or, equivalently, that 
\begin{equation}
\label{E:no-commute}
k_ik_jk_i^{-1}\neq k_j.
\end{equation}
Consider the circle bundle $\P_j\circ h_j^{-1}:L_{0}\to L_{\P_j(p_j)}$ where $\P_i$ is the map defined in \eqref{E:projection}. Equation \eqref{E:no-commute} says that $\P_j\circ h_j^{-1}$ is not orientable along a curve representing the class $k_i$. Therefore, we must have  
\begin{equation}\label{E:non_comm}
k_ik_jk_i^{-1}=k_j^{-1}
\end{equation}
or, equivalently, $[k_i,k_j]=k_j^{-2}$. From \eqref{E:non_comm} it follows that $k_j^2$ commutes with all $k_i$, and is therefore central in $K$.

By exchanging the roles of $k_i$ and $k_j$ in  \eqref{E:non_comm}, we similarly obtain $[k_i,k_j]=k_i^{-2}$.
Thus 
\[
k_i^{-2}=k_j^{-2}
\] 
for all non commuting $k_i,k_j$, $1\leq i,j\leq r$.
Since $k_i$ and $k_j$ do not commute,  $k_i$ and $k_j^{-1}$ do not commute either and therefore $k_j^2=k_i^{-2}=k_j^{-2}$. In particular, $k_i^4=e$ unless $k_i$ is central in $K$.

The center of $K$ splits uniquely as $Z(K)=Z_{(2)}\times K_1$, where $Z_{(2)}$ is the Sylow $2$-subgroup of $Z$. Let $N=\scal{k_i|\; k_i\notin Z(K)}$ and $K_2= N\cdot Z_{(2)}$.  By the computations above, $[K_2,K_2]$ is generated by the squares of the generators $k_i$ in $N$ and therefore it is central. Hence, it is abelian, and $K_2$ is $2$-step nilpotent. On the other hand, $[K_2,K_2]$ is finitely generated and every generator has order $2$. Therefore, $[K_2,K_2]$ is a finite $2$-group. It follows from the short exact sequence 
\[
0\rightarrow [K_2,K_2]\rightarrow K_2\rightarrow K_2/[K_2,K_2]\rightarrow 0
\]
that $K_2$ is a finite $2$-group. Clearly, $K_1$ is abelian, $[K_1,K_2]=0$ and $K_1\cap K_2=\{e\}$. Thus $K=K_1\times K_2$.
\\

\noindent \textit{Proof of Step 3.} The map $\partial$ in  sequence \eqref{E:Seq_rigid} can be seen as the map $\alpha_*: \pi_1(\Omega B,b_0)\to \pi_1(L_0,p_0)$ induced by the fibration $\Omega B\to L_0 \to \hat{M}_0$, where $\Omega B$ is the loop space of $B$. The map  $\alpha:\Omega B\to L_0$ extends to an action of $\Omega B$ on $L_0$ via the holonomy of the fibration $\hat{M}_0\to B$. We  will denote this action by ``$\,\star\,$''. The existence of this action implies that $\alpha_*(\pi_1(\Omega B,b_0))\simeq \partial (\pi_2(B))$ is central in $\pi_1(L_0,p_0)$. Indeed, given $\gamma_L\in \pi_1(L_0,p_0)$ and $\gamma_B\in \pi_1(\Omega B, b_0)$, the homotopy
\[
H(s,t)=\gamma_B(s)\star\gamma_L(t)
\]
is a homotopy between $\alpha_*(\gamma_B)\cdot \gamma_L$ and $\gamma_L\cdot \alpha_*(\gamma_B)$. \hfill $\square$
\\

The following corollary is a direct consequence of the proof of Theorem~A. 

\begin{cor}
\label{C:REF_COR}
Let $(M,\fol)$ be a closed  singular Riemannian foliation on a compact, simply connected Riemannian manifold $M$. Let $M_0$ be the regular part of $\fol$, let $L_0$ be a regular leaf and let $r$ be the number of codimension $2$ strata. If $\pi_2^{\mathrm{orb}}(M_0/\fol)=0$, then $\pi_1(L_0)$ is generated by at most $r$ elements.
\end{cor}

\begin{rem} In the proof of Theorem~\ref{T:fund_gp}, the non-abelian part $K_2$ is of a very particular type. One can prove that there exists a surjective homomorphism $\overline{K}\rightarrow K_2$, where $\overline{K}$  is  isomorphic to a finite product of groups $\overline{K}_i$ that are central extensions $1\to \mathbb{Z}_2\to \overline{K}_i\to \mathbb{Z}_2^{n_i}\to 1$, for some $n_i\geq 0$. The only instances known to us of such groups are $\mathbb{Z}_2$, $\mathbb{Z}_4$, the quaternion group $Q$ and products of these groups. 
\end{rem}

%-------------------------------------------------------------------------------------
%	SECTION:  EULER CHARACTERISTIC OF B-FOLIATIONS
%-------------------------------------------------------------------------------------

\section{Euler characteristic of B-foliations}
\label{euler}

\subsection*{Proof of Theorem~\ref{T:Kobayashi}}

% SUBSECTION: PROOF OF THEOREM B (KOBAYASHI'S THEOREM)

Let $W$ be a small tubular neighborhood  $B_\epsilon(\Sigma_0)$ of $\Sigma_0$ and let $V=M\setminus B_{\epsilon/2}(\Sigma_0)$. As $\{\, W,V \, \}$ is an open cover of $M$, we have 
\begin{eqnarray*}
	\chi(M)	& = & \chi(W)+\chi(V)-\chi(W\cap V)\\[.2cm]
			& = & \chi(\Sigma_0)+\chi(V)-\chi(W\cap V).
\end{eqnarray*}
Observe that $V$ and $W\cap V$ are saturated submanifolds without $0$-dimensional leaves. We will now show that $\chi(V\cap W)=\chi(V)=0$, which proves the theorem. Notice that $V\cap W$ retracts to $M'=\partial B_{\epsilon/2}(\Sigma_0)$, which is a compact saturated submanifold of $M$. 

Recall that $M'/\fol|_{M'}$ admits a finite good open cover $\{U_1^*,\ldots, U_k^*\}$,  i.e.~every finite intersection $U_{\alpha_1}^*\cap\ldots\cap U_{\alpha_k}^*$ is contractible 
(cf.~\cite{Wk}). For $1\leq i\leq k$, we let $U_i\subset M$ be the preimage of $U_i^*$ under the leaf projection map.  We can write
\[
\chi(M')=\sum_{i} \chi(U_i)-\sum_{i,j}\chi(U_i\cap U_j)+\ldots,
\]
where the sum is finite. Observe that every finite intersection retracts  to a leaf. Since these leaves are homeomorphic to nontrivial Bieberbach manifolds, their Euler characteristic is $0$. In particular, $\chi(M')=\chi(W\cap V)=0$. We will now show that $\chi(V)=0$. Observe first that $\overline{V}$, the closure of $V$, has boundary $M'$. The double $M''=\overline{V}\cup_{M'} \overline{V}$ is a compact manifold which admits a B-foliation without $0$-dimensional leaves and, as before, $\chi(M'')=0$. On the other hand, $\chi(M'')=2\chi(V)-\chi(M')$, so $\chi(V)=0$. 
\hfill $\square$

%--------------------------------------------------------------------------
%	SECTION: CODIM 1 A-FOLIATIONS 
%--------------------------------------------------------------------------

\section{A-foliations of codimension $1$ on simply connected manifolds}
\label{codim1}

\subsection*{Proof of Theorem~\ref{T:CODIM_1_BFOL_SC}}

We first prove the following lemma.

% LEM

\begin{lem}
\label{L:COD1_SING}
Let $M^{n+1}$ be a compact, simply connected $(n+1)$-manifold. If $(M^{n+1},\fol^n)$ is a codimension one closed singular Riemannian foliation, then the foliation cannot be regular. 
\end{lem}

% PROOF

\begin{proof}
Suppose that  $\fol^n$ is a regular foliation. Since  $M^{n+1}$ is simply connected, it follows from  work of Molino \cite{Mo} that $\fol^n$ must be a simple foliation,  i.e.~it is given by the fibers of a Riemannian submersion. Hence there is a fibration $L^n\to M^{n+1} \to \sphere^1$. Since $M^{n+1}$ is simply connected, the long exact sequence in homotopy for the fibration yields a contradiction. 
\end{proof}

% PROOF OF THM

We now prove Theorem~\ref{T:CODIM_1_BFOL_SC}. Let $M$ be a compact, simply connected manifold and $(M,\fol)$ a codimension $1$ A-foliation of $M$. By Lemma~\ref{L:COD1_SING},  $\fol$ is singular. Therefore, the leaf space $M^*$ is homeomorphic to a closed interval $[-1,1]$.  In particular, $M_0/\fol\cong (-1,+1)$ is a contractible manifold and there are at most two strata of codimension $2$. By Corollary~\ref{C:REF_COR}, the fundamental group of a regular leaf has at most two generators. By Corollary~\ref{C:torus_leaves}, the regular leaves are homeomorphic to tori. Therefore, a regular leaf must be diffeomorphic to $\sphere^1$ or $T^2$. It follows that $M$ is either $2$- or $3$-dimensional and, since $M$ is simply connected, it must be diffeomorphic to $\sphere^2$ or to $\sphere^3$. In the case of $\sphere^2$, it follows from Theorem~\ref{T:HOMOGENEITY} that the foliation comes from a smooth circle action. 

Suppose now that $M$ is diffeomorphic to $\sphere^3$. In this case, the regular leaves are $2$-dimensional, the singular leaves $L_{\pm}$ are $1$-dimensional, hence they are circles, and the regular leaves fiber over $L_{\pm}$ with fiber a circle. By Corollary~\ref{C:torus_leaves}, the regular leaves are  diffeomorphic  to a $2$-torus. We can therefore describe 
$\sphere^3$ as a double disk bundle
\[
M=D_-\cup_{\phi}D_+,
\]
where $D_{\pm}$  is a disk bundle over $L_{\pm}$ and $\phi:\partial D_+\ra \partial D_-$ is a diffeomorphism. Moreover, the foliation $\fol$ consists of the distance tubes to the zero section (with respect to some Euclidean structure on the disk bundles). Notice that $D_{\pm}$ are solid tori and $\partial D_{\pm}$ are  tori. 

Since $M$ is diffeomorphic to $\sphere^3$, the gluing map $\phi$ is unique up to isotopy. In particular, since the foliation $(M,\fol)$ is uniquely determined up to foliated diffeomorphism by the isotopy type of $\phi$. It follows that there is only one foliated diffeomorphism type of codimension one foliation in $\sphere^3$ that decomposes $\sphere^3$ into two full tori, and this must be the one given by the standard linear $T^2$-action. \hfill$\square$

%--------------------------------------------------------------------------
%	SECTION: CODIM 2 A-FOLIATIONS 
%--------------------------------------------------------------------------

\section{A-foliations of codimension $2$ on simply connected manifolds}
\label{S:CODIM_2}

\subsection*{Proof of Theorem~\ref{T:BFOL_CODIM2}}

Throughout this section we let $(M,\fol)$ be an A-foliation of codimension $2$ on a compact simply connected Riemannian manifold $M$.

% SUBSECTION: REGULAR A-FOLIATIONS

\subsection{Regular A-foliations of codimension $2$} We first consider  the case where $(M,\fol)$ is regular. 

% PROP

\begin{prop}
If $\fol$ is regular, then $M$ is diffeomorphic to $\sphere^3$ and the generic leaf is diffeomorphic to $\sphere^1$.
\end{prop}

% PROOF 

\begin{proof}
Since $(M,\fol)$ is regular and has codimension $2$, it follows from \cite[Theorem~1.6]{Ly10} that the quotient $M/\fol$ is a compact simply connected orbifold without boundary. In particular, $M/\fol$ is homeomorphic to $\sphere^2$.

In our case, the fibration \eqref{E:Borel_Molino} is given by $L\rightarrow M\rightarrow B$, where $L$ is a regular leaf of $\fol$ and $B$ is the classifying space of the orbifold $M/\fol$. Therefore, there is a rational homotopy equivalence $B\rightarrow M/\fol$ and, as a consequence, $\pi_2(B)\otimes\mathbb{Q}\cong \mathbb{Q}$. Tensoring the long exact homotopy sequence of the fibration $L\rightarrow M\rightarrow B$ with $\mathbb{Q}$, we get
\[
\mathbb{Q}\rightarrow \pi_1(L)\otimes\mathbb{Q}\rightarrow 0.
\]
By Corollary~\ref{C:torus_leaves}, $L$ is a torus. Therefore, the sequence above implies that $L$ is diffeomorphic to $\sphere^1$. Since $(M,\fol)$ has codimension $2$, $M$ is $3$-dimensional. By Perelman's proof of the Poincar\'e conjecture, $M$ must be diffeomorphic to $\sphere^3$.
\end{proof}

By Theorem~\ref{T:HOMOGENEITY}, a regular  foliation on $\sphere^3$ by circles must be homogeneous. Moreover, the circle action must be equivalent to a linear circle action on $\sphere^3$ (cf.~\cite{Or}).  

% SUBSECTION: SINGULAR B-FOLIATIONS OF CODIM 2

\subsection{Singular A-foliations of codimension $2$} 
We now consider the case where $(M,\fol)$ is singular. 
\\

Since the foliation is not regular, by  \cite{Ly10}  there are no exceptional leaves and the leaf space $M^*$ is homeomorphic to a $2$-dimensional orbifold $B$ with non-empty boundary corresponding to singular strata. As in the case of group actions, the fundamental group of $M$ surjects onto the fundamental group of the leaf space (cf.~\cite[Chapter~II, Theorem~6.2 and Corollary~6.3]{Br}). Therefore, the leaf space is simply connected  and hence it is homeomorphic to a disk. The boundary of $B$ consists of the union of geodesic arcs.  The points in the interior of these arcs correspond to leaves which we will call \emph{least singular leaves}, while the vertices of the leaf space, i.e.~the points where two geodesic arcs in the boundary meet, correspond to leaves which we will call \emph{most singular leaves} (see Figure~\ref{F:lspace_labels}).

% FIGURE

\begin{figure}
\psfrag{A}{Most singular leaf}
\psfrag{B}{Least singular leaf}
\psfrag{C}{Regular leaf}
\centering
\includegraphics[scale=1]{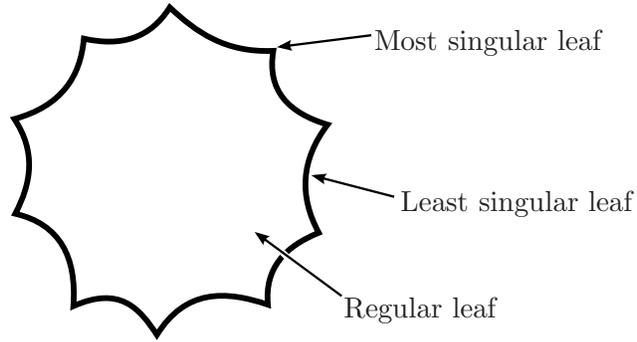}
\caption{Example of a leaf space of a codimension $2$ A-foliation.}
\label{F:lspace_labels}

\end{figure}

Let $L$ be a singular leaf and fix $p\in L$.  By Corollary~\ref{C:Inf_B_fol}, the infinitesimal foliation $(\sphere^r_p,\fol_p)$ is a codimension one A-foliation, whose quotient is a closed interval. By Theorem~\ref{T:CODIM_1_BFOL_SC}, the infinitesimal foliation $(\sphere^r_p,\fol_p)$ must be one of the homogeneous foliations $(\sphere^2,\sphere^1)$ or  $(\sphere^3,T^2)$. Since $\sphere^r_p$ is a round unit sphere, it follows from the main theorem in \cite{Ra} that the foliation is isometric to  one of the homogeneous foliations $(\sphere^2,\sphere^1)$ or  $(\sphere^3,T^2)$ induced by orthogonal actions on round unit spheres. In the first case, the leaf space of the homogeneous foliation is isometric to a closed interval of length $\pi$ and corresponds to the infinitesimal foliation of a least singular leaf. In the second case, the leaf space of the homogeneous foliation is isometric to a closed interval of length $\pi/2$ and corresponds to the infinitesimal foliation of a most singular leaf. 

Since there are no exceptional leaves, by Remark~\ref{R:Space_Dir}, the holonomy action is trivial. Hence the space of directions at any point $p^*\in M^*$  is isometric to the leaf space of  the infinitesimal foliation at $p^*$. In particular, the angle between geodesic arcs meeting at vertices of $M^*$ is $\pi/2$ and fibration~\eqref{E:Inf_fibration} yields the desired metric fibrations in part (2) of Theorem~\ref{T:BFOL_CODIM2}. Finally, since the regular leaves are homeomorphic to $n$-tori and the edges of the leaf space correspond to codimension $2$ strata, Corollary~\ref{C:REF_COR} implies that there must be at least $n$ edges. This concludes the proof of Theorem~\ref{T:BFOL_CODIM2}. \hfill $\square$

%------------------------------------------------------------------------------------------------------------------------------------
%	SECTION:  CURVATURE AND A-FOLIATIONS OF CODIM 2
%------------------------------------------------------------------------------------------------------------------------------------

\section{Curvature and A-foliations of codimension $2$}
\label{S:CURV_1}

% SUBSECTION: PROOF OF THEOREM 

\subsection{Proof of Theorem~\ref{T:QPC_D4_CODIM2}} 
\label{SS:Proof_4D_QPC}
We prove assertions (1) and (2) separately. 
\\

% PROOF OF  (1)

\noindent\textit{Proof of assertion (1).} Suppose $n=4$. Then the leaf space is a $2$-dimensional Riemannian manifold of nonnegative curvature, homeomorphic to a $2-$-disk, with polyhedral boundary and with positive curvature in an open subset. By Theorem~\ref{T:BFOL_CODIM2}, the leaf space $M^*$ has $m\geq 2$ vertices where boundary edges meet at an angle of $\pi/2$. These points, which we will denote by $p_i^*$, correspond to $0$-dimensional leaves $p_i$ in $M$. Since $M$ is quasipositively curved, so is $M^*$. By the Gauss-Bonnet theorem, $m\pi/2<2\pi\chi(M^*)=2\pi$, which implies that $m=2$ or $m=3$. In either case, $M$ decomposes as a union $M=D_1\cup_{\phi} D_2$, where $D_1$ is diffeomorphic to a small distance ball around $p_1$ and $D_2$ is diffeomorphic to a tubular neighborhood around either $p_2$, if $m=2$, or to the closure of the stratum opposite to $p_1$, if $m=3$. Therefore, $\partial D_1\simeq \partial D_2\simeq \sphere^3$. 

If $m=2$,  there exist diffeomorphisms $\psi_i:D_i\to B^4$ where $B^4$ is a unit ball in $\RR^4$, and $M$ is diffeomorphic to $B^4\cup_{\ol{\phi}}B^4$, where $\ol{\phi}=\psi_2|_{\partial D_2}\circ \phi \circ \psi_1|_{\partial D_1}^{-1}: \sphere^3 \to \sphere^3$. By Hatcher's proof of the Smale conjecture \cite{Ha}, $\phi$ is isotopic to $\ol{\phi}_0\in \mathrm{O}(4)$, and thus $M\simeq B^4\cup_{\ol{\phi}_0}B^4\simeq \sphere^4$.

If $m=3$, the closure $N$ of the stratum opposite to $p_1$ is a $2$-dimensional smooth submanifold of $M$ and $\fol$ restricts to a codimension one foliation with two singular leaves. Hence, $N$ is diffeomorphic to $\sphere^2$ and, since $\partial D_2\simeq \sphere^3$, there is a diffeomorphism  $\psi_2:D_2\to B$, where $B$ is a tubular neighborhood of a totally geodesic $\sphere^2$ in $\CC\PP^2$. Again, there is a diffeomorphism $\psi_1:D_1\to B^4$, and $M$ is diffeomorphic to $B^4\cup_{\ol{\phi}}B$ where $\ol{\phi}$ is defined using $\psi_1,\psi_2$, as before. Once again, $\ol{\phi}$ is isotopic to $\ol{\phi}_0\in \mathrm{O}(4)$, and $M\simeq B_4\cup_{\ol{\phi}} B\simeq\CP^2$.
\\

% PROOF OF  (2)

\noindent\textit{Proof of assertion (2).} By the work of Barden and Smale \cite{Ba, Sm}, it suffices to verify that $H_2(M,\mathbb{Z})=0$. The leaf space $M^*$ is homeomorphic to a disk and it has at least three  vertices, by Theorem~\ref{T:BFOL_CODIM2}. A comparison argument as in part (1) implies that there are exactly three vertices in $M^*$.

Let $X_-\In M$ be the preimage of an edge of $M^*$, and let $X_+=\sphere^1$ be the preimage of the opposite vertex. The preimage $X_-$ is smooth closed submanifold of $M$ without boundary, and it is a deformation retract of $M\setminus X_+$. Since $\codim(X_+)>2$, $\pi_1(X_-)=\pi_1(M\setminus X_+)=1$ and therefore $X_-=\sphere^3$. It follows that $M$ admits a decomposition as a double disk bundle
\[
M=D(X_-)\cup D(X_+),\qquad \partial D(X_-)=\partial D(X_+)=X_0
\]
where $X_0$ is a distance tube around $X_-$ and $X_0\to X_1$ is an orientable circle bundle. In particular, $X_0=\sphere^3\times \sphere^1$ and from the Meyer-Vietoris sequence applied to the double disk decomposition we obtain
\[
H_2(\sphere^3\times \sphere^1,\mathbb{Z})\to H_2(\sphere^3,\mathbb{Z})\to H_2(M,\mathbb{Z})\to \ZZ \to \ZZ\to 0,
\]
from which it follows easily that $H_2(M,\mathbb{Z})=0$.\hfill $\square$

% SUBSECTION: PROOF OF THEOREM E

\subsection{Proof of Theorem~\ref{T:NNC_D4_CODIM2}} 
\label{SS:NNEG_CURV_ERK}
We prove parts (1) and (2) separately.
\\

\noindent\textit{Proof of part (1).} 
By the Gauss-Bonnet Theorem, the leaf space $M^*$ has at most four vertices. On the other hand,  by Theorem~\ref{T:BFOL_CODIM2},  $M^*$ has at least two vertices. 
Hence, by Theorem~\ref{T:Kobayashi}, the Euler characteristic of $M^4$ is $2$, $3$ or $4$, and it follows from the  work of Freedman \cite{Fr} that $M^4$ is homeomorphic to $\sphere^4$, $\CP^2$, $\CP^2\#\pm\CP^2$ or $\sphere^2\times\sphere^2$.
 If $M^*$ has two or three vertices, then the leaf space structure is the same as the one in the proof of part (1) of Theorem~\ref{T:QPC_D4_CODIM2} and $M$ admits a decomposition as a double disk bundle. It follows that $M^4$ is diffeomorphic to $\sphere^4$, if $M^*$ has two vertices, or to $\CP^2$, if $M^*$ has three vertices. If $M^*$ has four vertices, then $M^*$ is isometric to a flat rectangle and $M$ also admits  a double disk bundle decomposition. In this case, it follows from Theorem~1.1 in \cite{GR} that $M$ is diffeomorphic to one of  $\CP^2\#\pm\CP^2$ or $\sphere^2\times\sphere^2$.
\\

\noindent\textit{Proof of part (2).} By Theorem~\ref{T:BFOL_CODIM2}, the leaf space $M^*$ has at least three vertices. On, the other hand, since the leaf space is nonnegatively curved, it may have at most four vertices. If there are exactly three, then, proceeding as in the proof of Theorem~\ref{T:QPC_D4_CODIM2}, we conclude that $M$ is diffeomorphic to $\sphere^5$. Therefore, we can restrict our attention to the case in which  $M^*$ has four vertices, so that $M^*$ is isometric to a flat rectangle $[-1,1]\times[-1,1]$. To conclude that $M$ is diffeomorphic to one of the two $\sphere^3$-bundles over $\sphere^2$, it suffices to prove that $H_2(M,\ZZ)=\ZZ$ and then appeal to the work of Barden \cite{Ba} and Smale \cite{Sm}.
\\

Let $X_{\pm}\In M$ and $X_0\In M$ be the preimage of $\{\pm 1\}\times [-1,1]\In M^*$ and $\{0\}\times [-1,1]$, respectively. Similarly, define $Y_{\pm},Y_0$ to be the preimages of $[-1,1]\times\{\pm 1\}$, $[-1,1]\times\{0\}$. They are all smooth closed submanifolds of $M$ without boundary, there are maps $\phi_{\pm}:X_0\to X_\pm$ which are circle bundles, and $M$ can be written as a double disk bundle
\begin{equation}\label{E:double_disk}
M=D(X_-)\cup D(X_+),\qquad \partial D(X_-)=\partial D(X_+)=X_0
\end{equation}
Let $N$ be the preimage of the point $(-1,-1)\in M^*$. The leaf $N$ is diffeomorphic to $\sphere^1$, and its normal bundle $\nu (N)$ has rank $4$ and is orientable. Since the universal cover of $N$ is contractible, the structure group of $\nu(N)$ is completely determined by an isometry $P$ of $\sphere^3_p=\nu^1_pN$ at some point $p\in N$. Since the structure group preserves the foliation $\fol$, the isometry $P$ preserves the infinitesimal foliation $(\sphere^3_p,\fol_p)$, which is isometric to the foliation induced by the isometric $T^2$ action on $\sphere^3$. Since $P$ also preserves the orientation of $\sphere^3$, $P\in \SO(4)\cap \Iso(\sphere^3_p,\fol_p)=\SO(4)\cap (\On(2)\times \On(2))=\mathrm{S}(\On(2)\times \On(2))$.

\begin{lem}
$P\in \SO(2)\times \SO(2)$.
\end{lem}
\begin{proof}
$P$ belongs to $\SO(2)\times\SO(2)$ if and only if it preserves the orientation of both singular leaves $\leaf_1, \leaf_2\simeq \sphere^1$ of the infinitesimal foliation at $p$, otherwise it reverses both. If $L_1,L_2\in \fol$ are the (singular) leaves containing $\exp_p\leaf_1$, $\exp_p\leaf_2$ respectively, they are 2-dimensional and they are both tori if and only if $P\in \SO(2)\times \SO(2)$, otherwise they are both Klein bottles. It follows that either all the 2-dimensional leaves are orientable (if $P\in \SO(2)\times \SO(2)$) or none of them is. Moreover, since the generic leaf in $X_{-}$ is orientable if and only if $X_-$ is (and the same holds for $X_+, Y_{\pm}$) it follows that $X_{\pm},Y_{\pm}$ are all orientable if and only if $P\in \SO(2)\times \SO(2)$, otherwise none of them is.

Suppose now that $P\notin \SO(2)\times \SO(2)$. As we said, it follows that $X_+,X_-$ are non-orientable, and since $X_0$ is always orientable, the circle bundles $\phi_{\pm}:X_0\to X_{\pm}$ are non-orientable. It follows from equation \eqref{E:double_disk} and \cite[Table 1.4]{GH} that $\pi_1(X_0)$ is finite. On the other hand, $X_0$ itself can be written as a double disk bundle
\[
X_0=D(K_{-})\cup D(K_{+}),\qquad \partial D(K_{-})=\partial D(K_{+})=T^3,
\]
where $K_\pm$ are Klein bottles. From the Mayer-Vietoris sequence, it follows that $\pi_1(X_0)$ cannot be finite, and this provides a contradiction.
\end{proof}
The following statements immediately follow from the proof of the lemma:
\begin{itemize}
\item Every $2$-dimensional leaf is a torus.
\item The manifolds $X_{\pm}$, $Y_{\pm}$ are orientable.
\item The bundles $X_0\to X_{\pm}$ are orientable, and therefore principal $\sphere^1$-bundles.
\end{itemize}
From the facts listed above, $X_-$ can be decomposed as a union of two solid tori. Therefore, $X_-$ is diffeomorphic to either $\sphere^2\times \sphere^1$ or to a lens space $L_m=\sphere^3/\ZZ_m$. Since $\phi_-:X_0\to X_-$ is a principal bundle, $X_0$ is homotopy equivalent to either $T^2\times \sphere^2$ (only if $X_-=\sphere^1\times \sphere^2$) or to $\sphere^1\times L_m$. If $X_0\sim S^1\times L_m$, consider the homotopy fibration $F\to X_0 \hookrightarrow M$. Since $\phi_{\pm}$ are orientable, it follows from \cite[Table 1.4]{GH} that $\pi_1(F)=\ZZ\oplus \ZZ$, and from the long exact sequence in homotopy we obtain
\[
0\to \pi_2(M)\to \ZZ\oplus\ZZ\to \ZZ\oplus \ZZ_m\to 0.
\]
Therefore $H_2(M,\mathbb{Z})=\pi_2(M)=\ZZ$.

The only possibility left is that $X_{0}\sim \sphere^2\times T^2$, and $X_{\pm}=\sphere^2\times \sphere^1$. Applying the Meyer-Vietoris sequence to the double disk bundle decomposition \eqref{E:double_disk} we obtain
\[
H_2(\sphere^2\times T^2)\stackrel{\Delta_*}{\lra} H_2(\sphere^2\times\sphere^1)\oplus H_2(\sphere^2\times\sphere^1)\to H_2(M)\stackrel{\partial_*}{\lra} \ZZ^2\to \ZZ^2\to 0.
\]
It follows immediately that $\partial_*=0$. Moreover, the map
 \[
 \Delta_*:H_2(\sphere^2\times T^2)=\ZZ^2\to H_2(\sphere^2\times\sphere^1)\oplus H_2(\sphere^2\times\sphere^1)=\ZZ^2
 \] 
is explicitly computable and its cokernel is $\ZZ$. Therefore $H_2(M, \mathbb{Z})=\ZZ$ in this case as well.
\hfill$\square$

%------------------------------------------------------------------------------------------------------------------------------------
%	SECTION:  FOLIATIONS IN THE PRESENCE OF CURVATURE II: FOLIATIONS BY CIRCLES
%------------------------------------------------------------------------------------------------------------------------------------

\section{Curvature and singular Riemannian foliations by circles}
\label{S:CURV_2}

\subsection*{Proof of Theorem~\ref{T:D4_S1FOL}}

Throughout this section we let $(M,\fol)$ be a singular Riemannian foliation by circles on a compact, simply connected Riemannian $4$-manifold. By Theorem~\ref{T:HOMOGENEITY}, $(M,\fol)$ is a homogeneous foliation, i.e.~it is induced by a smooth effective circle action on $M$. By work of  Fintushel \cite{F1,F2}, Pao \cite{Pa}, and Perelman's proof of the Poincar\'e conjecture, a compact, simply connected smooth $4$-manifold with a smooth effective circle action is diffeomorphic to a connected sum of copies of $\sphere^4$, $\pm\CP^2$ and $\sphere^2\times \sphere^2$.  It follows from Theorem~\ref{T:HOMOGENEITY} that a compact, simply connected $4$-manifold with a singular Riemannian foliation by circles is diffeomorphic to a connected sum of copies of $\sphere^4$, $\pm\CP^2$ and $\sphere^2\times \sphere^2$.

 The leaf space structure of $(M,\fol)$ corresponds to the orbit space structure of a smooth circle action on a compact, simply connected smooth $4$-manifold (cf.~\cite{F1}). In particular, the leaf space $M^*$ is a simply connected topological $3$-manifold, possibly with boundary, the components of the $0$-dimensional stratum are homeomorphic to $2$-spheres or isolated points and the boundary components of $M^*$ are $2$-sphere components in the $0$-dimensional stratum. With these preliminary remarks in place, we prove the rest of Theorem~\ref{T:D4_S1FOL}.

% SUBSECTION: PROOF OF MAIN THM IN THIS SECTION

\subsection*{Proof of  (1) and (2) of Theorem~\ref{T:D4_S1FOL}}
\label{SS:Pf_Thm_4DS1FOL}

By Poincar\'e duality, $2\leq \chi(M)$. 
By the discussion in the first paragraph of the proof, 
to prove parts (1) and (2) of the theorem it suffices to show that $\chi(M)\leq 3$, when $M$ is positively curved, and $\chi(M)\leq 4$, when $M$ is nonnegatively curved.

By the results in the preceding subsection, the leaf space $M^*$ is a simply connected topological manifold with an Alexandrov space structure. In particular, $M^*$ is positively curved if $M$ has positive curvature, and $M^*$ is nonnegatively curved if $M$ has nonnegative curvature. Since the proof follows as in the proof of an isometric circle action, via comparison arguments already found in the literature (cf.~\cite{HK, GS01,SY,GW}), we only indicate the necessary steps.
\\

\noindent\textit{Positively curved case}. Suppose first that  $\partial M^*$ is  not empty, and let $F^*$ be a  connected component of $\partial M^*$. Then, by the Soul Theorem for Alexandrov spaces, there exists a unique  point $p_0^*$ at maximal distance from $F^*$ and all the points between $F^*$ and $p_0^*$ must correspond to regular leaves. In particular, the boundary of $M^*$ is always connected. The point $p_0^*$ is either a regular leaf or an isolated point in $\Sigma_0$. Therefore $\chi(M)\leq 3$. Suppose now that $\partial M^*$ is empty. Then $\Sigma_0$ consists only of isolated points. The space of directions at an isolated point in the $0$-dimensional stratum is isometric to the quotient of a round unit $\sphere^3$ by an isometric circle action without fixed points. By a triangle comparison argument as in \cite{GW}, there can be at most three such points. 
Therefore, $\chi(\Sigma_0)\leq 3$.
\\

\noindent\textit{Nonnegatively  curved case.} Suppose first that $\partial M^*$ has at least two components, $\partial M^*_{-}$ and $\partial M^*_{+}$. Then $M^*$ is isometric to $\partial M^*_+\times [0,1]$ and there are no isolated points in $\Sigma_0$. Hence $\chi(M)= 4$. Suppose now that $\partial M^*$ is connected and let $C^*$ be the set at maximal distance from $\partial M^*$ in $M^*$. There can be at most $2$ isolated points in $\Sigma_0$ contained in $C^*$, so $\chi(M^*)\leq 4$. Finally, suppose that $\partial M^*$ is empty and $\Sigma_0$ consists only of isolated points. As in the positively curved case, a triangle comparison argument as in \cite{GW} implies that there can be  at most four such points, so $\chi(M)\leq 4$.\hfill $\square$

%--------------------------------------------------------------------------
% 				BIBLIOGRAPHY
%--------------------------------------------------------------------------

\bibliographystyle{amsplain}

% ----------------------------------------------------------------
% END BIBLIOGRAPHY-------------------------------------------------
% ----------------------------------------------------------------

\end{document}